\documentclass[12pt]{amsart}
\usepackage{mathtools}

\DeclarePairedDelimiter{\floor}{\lfloor}{\rfloor}
\usepackage[T1]{fontenc}
\usepackage{lmodern}
\usepackage{microtype}
\usepackage{amsmath,amssymb,amsthm,mathtools}
\usepackage{enumitem}
\usepackage{aliascnt}
\usepackage{hyperref}
\usepackage[nameinlink,capitalize,noabbrev]{cleveref}

\hypersetup{
  colorlinks=true,
  linkcolor=blue,
  citecolor=blue,
  urlcolor=blue,
  pdftitle={Ballot lifts of multicomplexes and a real-rootedness problem of Brenti},
  pdfauthor={Jiaji Liu}
}

\newtheorem{theorem}{Theorem}[section]

\newaliascnt{proposition}{theorem}

\aliascntresetthe{proposition}

\newaliascnt{lemma}{theorem}

\aliascntresetthe{lemma}

\newaliascnt{corollary}{theorem}
\newtheorem{corollary}[corollary]{Corollary}
\aliascntresetthe{corollary}

\newaliascnt{conjecture}{theorem}

\aliascntresetthe{conjecture}

\newaliascnt{problem}{theorem}

\aliascntresetthe{problem}

\theoremstyle{definition}
\newaliascnt{definition}{theorem}

\aliascntresetthe{definition}

\newaliascnt{example}{theorem}

\aliascntresetthe{example}

\theoremstyle{remark}
\newaliascnt{remark}{theorem}
\newtheorem{remark}[remark]{Remark}
\aliascntresetthe{remark}

\crefname{theorem}{Theorem}{Theorems}
\Crefname{theorem}{Theorem}{Theorems}
\crefname{proposition}{Proposition}{Propositions}
\Crefname{proposition}{Proposition}{Propositions}
\crefname{lemma}{Lemma}{Lemmas}
\Crefname{lemma}{Lemma}{Lemmas}
\crefname{corollary}{Corollary}{Corollaries}
\Crefname{corollary}{Corollary}{Corollaries}
\crefname{conjecture}{Conjecture}{Conjectures}
\Crefname{conjecture}{Conjecture}{Conjectures}
\crefname{problem}{Problem}{Problems}
\Crefname{problem}{Problem}{Problems}
\crefname{definition}{Definition}{Definitions}
\Crefname{definition}{Definition}{Definitions}
\crefname{example}{Example}{Examples}
\Crefname{example}{Example}{Examples}
\crefname{remark}{Remark}{Remarks}
\Crefname{remark}{Remark}{Remarks}

\newcommand{\B}{\mathcal{B}}

\newcommand{\N}{\mathbb N}

\newcommand{\Lcal}{\mathcal L}
\newcommand{\degm}{\operatorname{deg}}

\begin{document}

\title{Palindromic real-rooted polynomials and $h$-vectors of polytopes}

\author{Peter L. Guo}
\address[Peter L. Guo]{Center for Combinatorics, LPMC, 
Nankai University, Tianjin 300071, P.R. China}
\email{lguo@nankai.edu.cn}

\author{Mingyang Kang}
\address[Mingyang Kang]{Center for Combinatorics, LPMC, 
Nankai University, Tianjin 300071, P.R. China}
\email{2120240005@mail.nankai.edu.cn}

\keywords{palindromic real-rooted polynomial, simplicial polytope, $h$-vector, $\gamma$-vector, multicomplex, $O$-sequence, $g$-theorem}

\subjclass[2020]{ 52B05,  05E45,    13F55}

\begin{abstract}
We prove that every monic palindromic 
real-rooted polynomial with nonnegative integer coefficients  is the $h$-polynomial of a simplicial
convex polytope. This resolves in the affirmative  a longstanding problem circulated   by  Brenti  since 2004. 
\end{abstract}

\maketitle

\section{Introduction}

A polynomial of degree $d$ with nonnegative integer coefficients
\[
  h(t)=\sum_{i=0}^{d} h_i t^i\in \N[t]
\]
 is \emph{palindromic} (also called  symmetric) if $h_i=h_{d-i}$ for every $0\leq i\leq d$.

We provide  a positive  answer   to the  following longstanding  problem  which  has been informally  circulated   by Brenti since 2004  and formally proposed as \cite[Problem 2.9]{Brenti2024}. 

\begin{theorem}\label{prob:brenti}
 Let $h(t)=\sum_{i=0}^{d} h_i t^i\in \N[t]$ be a  monic palindromic  polynomial with only real roots.  Then  there exists  a
simplicial convex polytope whose $h$-vector is $(h_0,h_1,\ldots,h_d)$. 
\end{theorem}

We recall the $h$-vector of a simplicial convex polytope. 
A convex polytope is the convex hull of finitely many points  in a  Euclidean space.  Let $P$ be a simplicial convex polytope of dimension $d$, that is,   every proper  face of $P$ is a simplex. 
The boundary $\Delta(P)$ of $P$ is a simplicial complex consisting of all   proper faces of $P$. Here the empty set is considered as a proper face of dimension $-1$. For $0\leq i\leq d$, let $f_i(\Delta(P))$ denote the number of faces in $\Delta(P)$ with $i$ vertices\footnote{
In the literature, 
$f_i$
  usually denotes the number of 
$i$-dimensional faces. We adopt this convention to match the notation we will use later for multicomplexes.}, or equivalently, the number of faces of dimension $i-1$. 
The $h$-vector $h(P)=(h_0, h_1,\ldots, h_d)$ of $P$ is determined by 
\[
h_0+h_1t+\cdots+h_dt^d=\sum_{i=0}^d f_{i}(\Delta(P))\,t^i(1-t)^{d-i},
\]
which is called the $h$-polynomial of $P$.

\begin{remark}
The converse of   Theorem \ref{prob:brenti}  is not true. That is, the $h$-polynomial of a simplicial polytope $P$ is not necessarily  a monic palindromic polynomial with only real roots. For example, let $P$ be the simplex of dimension $d=2$. Then $f(\Delta(P))=(1, 3,3)$ and $h(P)=(1,1,1)$. It is clear that $1+t+t^2$ is not real-rooted. 
\end{remark}

We delineate the principal strategy   that will 
be employed in the proof of   \cref{prob:brenti}. 
The first result we use is  the famous $g$-theorem which gives a characterization of 
$h$-vectors  in terms of $O$-sequences. 
A finite nonempty set $M$ of monomials
is a \emph{multicomplex}  (also called order ideal of monomials  \cite[Section 2]{Stanley-2}) if whenever   
$ u\in M$ and $v$ divides  $u$, then   $ v\in M$.  
Clearly, a simplicial  complex corresponds to the case of squarefree monomials. 
For $i\geq 0$, write 
\[M_i=\{u\in M:\degm u=i\}.\] 
Let  $f_i(M)=|M_i|$.   A multicomplex is   of degree $s$ 
if the maximum of degrees of its monomials  is equal to $s$, and is   pure  if all maximal monomials have the same degree. 
A  sequence $(a_0,a_1,\ldots,a_s)$ is an  \emph{$O$-sequence} (also called an $M$-vector \cite[Page 56]{Stanley1996}) if there is a   multicomplex $M$ such that
$f_i(M)=a_i$ for  $0\leq i\leq s$, and is called a pure $O$-sequence if the  multicomplex $M$ is pure. 

Let  $h=(h_0,\ldots, h_d)$ be a palindromic integer vector. The associated $g$-vector  $g(h)=(g_0,\ldots, g_{\left\lfloor d/2\right\rfloor})$  is defined   by letting
\[
  g_0=h_0,
  \qquad
  g_i=h_i-h_{i-1}
  \quad \text{ for $1\leq i\leq \left\lfloor d/2\right\rfloor$
}.
\]
Billera and Lee \cite{BilleraLee1981}  and Stanley \cite{Stanley1980} respectively proved the sufficiency and necessity of the following deep  relation conjectured by McMullen \cite{M-1, M-2} (which is now known as the $g$-theorem): a palindromic integer  vector   $h$  is the $h$-vector of a
simplicial $d$-dimensional convex polytope if and only if $g_0=1$, $g_1\geq 0$ and $0\leq g_{i+1}\leq g_{i}^{<i>}$ for $1\leq i\leq \left\lfloor d/2\right\rfloor-1$. 
Stanley \cite[Section 2.2]{Stanley-2} (see also \cite[Page 183]{BL80})   showed that the condition on the right-hand side is equivalent to the statement  that $g(h)$ is an $O$-sequence. Collecting the above, we have 

\begin{theorem}[\text{\cite{BilleraLee1981,Stanley1980,Stanley-2}}]\label{gtheorem}
Let  $h=(h_0,h_1,\ldots, h_d)$ be a palindromic integer  vector. Then $h$   is the $h$-vector of a
simplicial $d$-dimensional convex polytope if and only if $g(h)$ is an    $O$-sequence. 
\end{theorem}

The second  result that we need  is due to Bell and Skandera
\cite{BellSkandera2007}.

\begin{theorem}[\text{\cite[(3.2) and Theorem 3.6]{BellSkandera2007}}]
\label{thm:bell-skandera}
Let
\[
  A(t)=1+a_1t+\cdots+a_rt^r\in\N[t]
\]
have only real zeros.  Then the coefficient sequence $(1,a_1,\ldots,a_r)$ is an $O$-sequence. 
\end{theorem}

We next need the real-rootedness   concerning  the $\gamma$-polynomial of a real-rooted palindromic polynomial. 
It is known that a palindromic polynomial $  h(t)=\sum_{i=0}^{d} h_i t^i$ has a unique linear  expression  in terms of the basis $\{  t^j(1+t)^{d-2j}\colon  0\leq j\leq \lfloor d/2\rfloor\}$:
\begin{equation*}
  h(t)=\sum_{j=0}^{\lfloor d/2\rfloor}
  \gamma_j\, t^j(1+t)^{d-2j}.
\end{equation*}
The vector $\gamma(h)=(\gamma_0,\gamma_1,\ldots, \gamma_{\lfloor d/2\rfloor})$
is called the $\gamma$-vector of $h(t)$. Write 
\[
  \gamma_h(t)=\sum_{j=0}^{\lfloor d/2\rfloor}\gamma_j\, t^j
\]
for the $\gamma$-polynomial associated to $h(t)$. 

The following theorem can be directly deduced from Gal \cite{Gal}. 

\begin{theorem}\label{prop:real-rooted-gamma}
 Let $h(t)=\sum_{i=0}^{d} h_i t^i\in \N[t]$ be a polynomial satisfying the conditions  in Theorem  \ref{prob:brenti}.  
Then $\gamma_h(t)$ 
belongs to $\N[t]$, has constant term $1$, and has only real negative
zeros.
\end{theorem}

\begin{proof}
Since $h(t)$ is a monic palindromic polynomial with  integer coefficients, by \cite[Proposition 2.1.1]{Gal} and \cite[Remark 2.1.3]{Gal},   $\gamma_h(t)$ is a polynomial with  integer coefficients and with   constant term equal to $1$. 

Since all coefficients of $h(t)$
are nonnegative and $h(0)=1$,  all the  real zeros of $h(t)$ are  negative. 
By  \cite[Remark 3.1.1]{Gal}, $\gamma_h(t)$ has only real  negative zeros. This implies that $\gamma_h(t)$ has nonnegative coefficients, thus belonging to  $\N[t]$.
\end{proof}

Combining Theorem \ref{prop:real-rooted-gamma} with Theorem
\ref{thm:bell-skandera} yields that  

\begin{corollary}\label{skci-09}
  Let $h(t)=\sum_{i=0}^{d} h_i t^i\in \N[t]$ be a polynomial  satisfying the conditions  in Theorem  \ref{prob:brenti}. Then $\gamma(h)$ is an $O$-sequence.    
\end{corollary}

Lastly, we establish the following result which allows us to complete the proof of Theorem  \ref{prob:brenti}.

\begin{theorem}\label{thm:intro-lift}
Let $h(t)=\sum_{i=0}^d h_i t^i$ be a palindromic polynomial.    If
$\gamma(h)$ is an $O$-sequence, then
$g(h)$ is a pure $O$-sequence.
\end{theorem}

Clearly, Corollary \ref{skci-09}, Theorem \ref{thm:intro-lift} and 
Theorem \ref{gtheorem} together lead to a proof of Theorem  \ref{prob:brenti}. 
 
\begin{remark}
By definition, the $f$-vector of a simplicial convex polytope is an $O$-sequence.  In this case when   $\gamma(h)$ is a $f$-vector, Theorem \ref{thm:intro-lift} has appeared as \cite[Proposition 6.4]{NPT11}, see also  \cite[Proposition 8.2]{CoronFerroniLi2026}. Therefore Theorem \ref{thm:intro-lift} can be thought of as a  multicomplex analogue of \cite[Proposition 6.4]{NPT11}. 
\end{remark}

\subsection*{Acknowledgements}
The first author was  supported by the National Natural Science Foundation of China (No. 12371329) and the Fundamental Research Funds for the Central Universities (No. 63263094).

\subsection*{Declaration of AI usage}
During the development of this work, the authors used ChatGPT  as a research-assistance tool for exploring possible proof strategies. The resulting arguments were substantially  reconstructed, independently verified (and in particular significantly simplified), and written by the authors.  

\section{Proof of Theorem \ref{thm:intro-lift}}

In this section, we shall give a proof of  Theorem \ref{thm:intro-lift}, which, as explained in the introduction,  allows us to conclude  Theorem  \ref{prob:brenti}.

For $n\geq 0$, define a set of squarefree monomials
\[
 \B_n
 =\bigl\{x_{s_1}\cdots x_{s_k}:
 1\leq s_1<\cdots<s_k\leq n,\quad s_r\geq 2r
 \text{ for }1\leq r\leq k\bigr\}.
\]
The empty product is regarded as an element of $\B_n$. 

\begin{remark}
  The set $\B_n$ is essentially the same as the simplicial complex  defined based on a ballot path  in \cite[Section 6]{NPT11}. Construct a path of length $n$ in east steps $E=(1,0)$ and
north steps $N=(0,1)$, and place an $N$ in positions
$s_1<\cdots<s_k$.  The inequalities $s_r\geq 2r$ are equivalent to the
condition that  the path never goes above the diagonal $y=x$.  
\end{remark}

We collect some properties concerning $\B_n$ observed  in  \cite[Section 6]{NPT11}. 
\begin{itemize}
    \item $\B_n$ is a pure finite simplicial complex of degree $\floor{n/2}$. 

    \item For $0\leq i\leq \floor{n/2}$,  there is the following formula   \begin{equation}\label{eq:ballot-count-1}
 f_i(\B_n)=\binom{n}{i}-\binom{n}{i-1}, \ \ \text{where $\binom{n}{-1}=0$.}
\end{equation}

\item  The following  nesting relation is clear:
\begin{equation}\label{eq:nesting}
 \B_n\subseteq \B_{n+1}.
\end{equation}
\end{itemize}
The complex $\B_n$ will be called a  ballot complex.

 We can now provide  a proof of 
Theorem \ref{thm:intro-lift}.

\begin{proof}[Proof of Theorem \ref{thm:intro-lift}]
Since $\gamma(h)$ is an $O$-sequence, we may choose a finite multicomplex $M$ in variables $y_1,\ldots,y_q$ such that
\[
 f_j(M)=\gamma_j, \ \ 0\leq j\leq \floor{d/2}.
\]
Define
\begin{equation}\label{eq:lift}
 \Lcal_d(M)
 =\bigl\{uv:u\in M,\ v\in \B_{d-2\deg u}\bigr\},
\end{equation}
where the $x$-variables used in the ballot complexes are taken to be disjoint from
the $y$-variables used in $M$.  We remark that when $M$ is squarefree (namely, a simplicial complex), the construction  $\Lcal_d(M)$ has appeared in  \cite[Proposition 6.4]{NPT11}.

We proceed to   prove  that
\begin{itemize}
    \item The family $\Lcal_d(M)$  is a finite pure  multicomplex. 
\end{itemize}
Let us first verify that $\Lcal_d(M)$ is closed
under divisibility, and so is a multicomplex.
  Let
$uv\in\Lcal_d(M)$. Suppose  that $w$ divides $ uv$. Then we have a unique factorization $w=u'v'$, where $u'$ uses only the
$y$-variables and $v'$ uses only the $x$-variables. It is evident that $ u'\mid u$ and $ v'\mid v.$
Hence $u'\in M$, and 
$v'\in\B_{d-2\deg u}$.  This implies $ \deg u'\leq \deg u$, and thus $ d-2\deg u\leq d-2\deg u'$. 
In view of  the nesting relation in \eqref{eq:nesting}, we are  given that 
\[
 v'\in\B_{d-2\deg u'}.
\]
This justifies  that  $w=u'v'$ also belongs to $\Lcal_d(M)$.

We next check that $\Lcal_d(M)$ is  pure.  Take $uv\in\Lcal_d(M)$.  Since
$\B_{d-2\deg u}$ is a pure complex of degree $\floor{(d-2\deg u)/2}$, there is a monomial
$\overline v\in\B_{d-2\deg u}$ such that $v$ divides $\overline v$ with 
\[ 
 \deg\overline v=   \left\lfloor \frac{d-2\deg u}{2}\right\rfloor.
\]
Notice that  $uv\mid u\overline v\in\Lcal_d(M)$ and
\[
 \deg(u\overline v)
 =\deg u+\left\lfloor \frac{d-2\deg u}{2}\right\rfloor
 =\left\lfloor  d/2\right\rfloor.
\]
Note also that  every $u_0v_0\in\Lcal_d(M)$ satisfies
\[
 \deg(u_0v_0)
 \leq \deg u_0+\left\lfloor\frac{d-2\deg u_0}{2}\right\rfloor=\left\lfloor  d/2\right\rfloor.
\]
Hence every monomial  in $\Lcal_d(M)$ divides a monomial of degree $\left\lfloor  d/2\right\rfloor$, and no
monomial in $\Lcal_d(M)$ has larger degree than $\left\lfloor  d/2\right\rfloor$. So  $\Lcal_d(M)$ is a pure multicomplex of  degree $\left\lfloor  d/2\right\rfloor$.

We complete the proof by showing that   the $g$-vector is the  $O$-sequence of the pure multicomplex $\Lcal_d(M)$.
Recall that each  degree-$i$ monomial in $\Lcal_d(M)$ has a unique factorization
$uv$ with $u\in M$ and $  v\in \B_{d-2\deg u}$.  Using the counting formula in 
\eqref{eq:ballot-count-1} and the fact that $ f_j(M)=\gamma_j$, we obtain that for $0\leq i\leq \left\lfloor  d/2\right\rfloor$,
\begin{align*}
 f_i\bigl(\Lcal_d(M)\bigr)
 &=\sum_{j=0}^{i} f_j(M)f_{i-j}(\B_{d-2j})\\
 &=\sum_{j=0}^{i}\gamma_j
 \left(
   \binom{d-2j}{i-j}-\binom{d-2j}{i-j-1}
 \right).
\end{align*}
The following relation has been  observed in \cite[Observation~6.2]{NPT11}: 
\begin{equation}\label{eq:gamma-to-g}
 g_i=\sum_{j=0}^{i}\gamma_j
 \left(
   \binom{d-2j}{i-j}-\binom{d-2j}{i-j-1}
 \right), \ \ \ 
\text{$0\leq i\leq \floor{d/2}$}.
\end{equation}
Therefore we have  
$g_i=f_i(\Lcal_d(M))$ for $0\leq i\leq \left\lfloor  d/2\right\rfloor$, as desired.   
\end{proof}

\end{document}